\documentclass{amsart} 
\usepackage{float}
\usepackage{newtxtext,newtxmath}

\usepackage{mathtools}

\usepackage{tikz}
\usetikzlibrary{arrows.meta, positioning, shapes.callouts, backgrounds, decorations.pathmorphing}
\usepackage{aliascnt}

\newtheorem{theorem}{Theorem}[section]

\newaliascnt{lemma}{theorem}
\newtheorem{lemma}[lemma]{Lemma}
\aliascntresetthe{lemma}

\newaliascnt{proposition}{theorem}
\newtheorem{proposition}[proposition]{Proposition}
\aliascntresetthe{proposition}

\newaliascnt{corollary}{theorem}
\newtheorem{corollary}[corollary]{Corollary}
\aliascntresetthe{corollary}

\newaliascnt{definition}{theorem}
\newtheorem{definition}[definition]{Definition}
\aliascntresetthe{definition}

\theoremstyle{remark}
\newaliascnt{remark}{theorem}
\newtheorem{remark}[remark]{Remark}
\aliascntresetthe{remark}

\newaliascnt{example}{theorem}
\newtheorem{example}[example]{Example}
\aliascntresetthe{example}

\usepackage{xcolor}
\usepackage{hyperref}

\usepackage[shortlabels]{enumitem}
\setlist[enumerate]{label=(\arabic*), itemjoin=\qquad}

\newcommand{\Net}[3]{\ensuremath{\mathrm{Net}_{#1,#2,#3}}}
\newcommand{\Raw}[3]{\ensuremath{\mathrm{Raw}_{#1,#2,#3}}}
\newcommand{\Xval}[2]{\ensuremath{X_{#1,#2}}}
\newcommand{\muval}[2]{\ensuremath{\mu_{#1}(#2)}}

\newcommand{\minval}[2]{\ensuremath{\min(#1, #2)}}
\newcommand{\floor}[1]{\left\lfloor #1 \right\rfloor}

\newcommand{\Fl}[1]{\mathrm{Fl}(\C^{#1})}
\newcommand{\SL}[1]{\mathrm{SL}_{#1}(\C)}

\newcommand{\fl}[1]{\lfloor#1\rfloor}
\newcommand{\bra}[1]{\left(#1\right)}
\newcommand{\cbra}[1]{\left\{#1\right\}}

\newcommand{\Mbox}{%
{\setlength{\fboxrule}{1.5pt}\fbox{2002}}%
}

\newcommand{\Lbox}{%
{\setlength{\fboxrule}{1.5pt}\fbox{20}}%
}

\newcommand{\Rbox}{%
{\setlength{\fboxrule}{1.5pt}\fbox{02}}%
}

\newcommand{\kbox}{%
{\fbox{1}}%
}

\newcommand{\nkbox}{%
{\fcolorbox{black}{lightgray}{1}}%
}

\def\a{\alpha}
\def\w{\varpi}

\def\S{\mathfrak{S}}

\def\Z{\mathbb{Z}}
\def\Q{\mathbb{Q}}

\def\C{\mathbb{C}}

\numberwithin{equation}{section}

\title{Products of Chern Classes and Chern Numbers on~the~Permutohedral~Variety}
\author{Hideya Kuwata}

\address{Department of Integrated Systems Engineering, Kindai University Technical College
\\ 
}

\email{hideya0813@gmail.com}

\begin{document}
\begin{abstract}
\noindent
A root system $\Phi$ of rank $n$ determines an $n$-dimensional smooth projective toric variety $X(\Phi)$ associated with the fan of its Weyl chambers. For the root system of type $A_n$, this variety is the well-known permutohedral variety $X_{A_n}$. Using purely combinatorial methods, we obtain an explicit closed formula expressing the product of Chern classes $c_k c_{n-k}$ as a multiple of the top Chern class $c_n$ in the rational cohomology ring $H^*(X_{A_n};\Q)$. The resulting coefficient, which depends only on $k$ and $n$, is given by a closed-form expression. As an application, we compute the Chern number~$\langle c_k c_{n-k}, [X_{A_n}] \rangle$.
\end{abstract}
\maketitle

\section{Introduction}
The $n$-dimensional permutohedral variety $X_{A_n}$ is a smooth projective toric variety associated with the fan of Weyl chambers for the root system of type $A_n$. This variety is a central object in algebraic combinatorics, with deep connections to the representation theory of symmetric groups 
\cite{Procesi1990,Stembridge1994} and an analog of Schubert calculus \cite{Abe2015}. This variety plays an important role in recent developments; for instance, it serves as a key example of a regular semisimple Hessenberg variety~\cite{DeMariProcesiShayman1992} and plays a central role in the study of log-concavity for matroids~\cite{HuhKatz2012}.

The cohomology of the permutohedral variety $X_{A_n}$ is determined by the combinatorial structure of the root system of type $A_n$. Specifically, the total Chern class is given by the product $\prod_{i=1}^n (1+[\alpha_i])$, where $[\alpha_i]$ is the cohomology class in $H^*(X_{A_n};\Q)$ associated with each simple root $\alpha_i$~\cite{Klyachko1995}. Expanding the product for the total Chern class, we see that the $k$-th Chern class $c_k$ is the $k$-th elementary symmetric function in the classes $[\alpha_1], \dots, [\alpha_n]$. In particular, the top Chern class is their product $c_n = [\alpha_1]\cdots[\alpha_n]$.
The pairing of $c_n$ with the fundamental homology class $[X_{A_n}]$ yields the Euler characteristic of $X_{A_n}$, which is~ $(n+1)!$.
In this paper, we study the product of the Chern classes $c_k$ and $c_{n-k}$. We express this product as an explicit multiple of the top Chern class $c_n$ and provide a combinatorial formula for the coefficient.

\begin{theorem}[\autoref{thm:MainTheorem}]
Let $k$ be an integer satisfying $0\leq k\leq n$. The product of Chern classes $c_k c_{n-k}$ is given by the relation
\[
c_kc_{n-k}=\mu_k(n)c_n,
\]
where the coefficient $\mu_k(n)$ is
\begin{equation}
    \mu_k(n)=\sum_{j=0}^{\mathclap{\lfloor k/2\rfloor}}\left(\frac{1}{12}\right)^j\binom{k-j}{j}\binom{n-k-j}{j}.
\end{equation}
\end{theorem}

As a direct result, pairing this class with the fundamental homology class $[X_{A_n}]$ yields the following formula for the Chern number.

\begin{corollary}[\autoref{cor:Chern-number}]
The Chern number $\langle c_k c_{n-k}, [X_{A_n}] \rangle$ is given by
\[
\langle c_k c_{n-k}, [X_{A_n}] \rangle=(n+1)!\mu_k(n).
\]
\end{corollary}

The remainder of this paper is organized as follows. In Section~2, we review the necessary background on the permutohedral variety $X_{A_n}$ and its cohomology ring. Sections~3 and~4 are devoted to developing the combinatorial framework required for our main result. Finally, the proof of the main theorem is presented in Section~5.

\section{Preliminaries}
\subsection{The Permutohedral Variety for Type \texorpdfstring{$A_n$}{An}}
Let $\{t_1,\dots, t_{n+1}\}$ be the standard basis in~$\mathbb{R}^{n+1}$ with the usual inner product, and consider the hyperplane 
\begin{equation}
\label{eq:Euclidean-Space}
E = \cbra{(x_1,\dots,x_{n+1}) \in \mathbb{R}^{n+1} ~\bigg|~ \sum_{i=1}^{n+1} x_i = 0}.    
\end{equation}
The root system $\Phi$ of type $A_n$ is realized in $E$ by the set of vectors 
$$
\Phi = \{t_i - t_j\in E \mid 1 \le i, j \le n+1, i \ne j\}.
$$
 We choose the simple roots to be $\alpha_i = t_i - t_{i+1}$ for $i=1, \dots, n$. For the root system of type $A_n$, the Weyl group is the symmetric group $W = \S_{n+1}$, which acts on the space $E$ by permuting the coordinates.
 
Let $M\subset E$ be the root lattice of $\Phi$ and $N\subset E^*$ be the coweight lattice of $\Phi$, where $E^*$ is the dual space of~$E$. Then $M$ is the dual lattice of $N$ with respect to the natural pairing~$\langle -,- \rangle$. The fundamental coweights $\w^{\vee}_1, \dots, \w^{\vee}_n$ in $E^*$ form the basis dual to the simple roots, defined by the relations $\langle \w^{\vee}_i, \alpha_j \rangle = \delta_{ij}$. The set of all coweights of the root system is given by the Weyl group orbit of the fundamental coweights, $\Phi^* = \bigcup_{u \in \S_{n+1}} \{u\w^{\vee}_1, \dots, u\w^{\vee}_n\}$.
The set of all coweights $\Phi^*$ generates a complete nonsingular fan in $E^*$, which we call the permutohedral fan and denote by $\Delta_{A_n}$. The maximal cones of $\Delta_{A_n}$ are precisely the Weyl chambers.
 The $n$-dimensional permutohedral variety, which we denote by $X_{A_n}$, is the smooth projective toric variety associated with the fan $\Delta_{A_n}$.

\subsection{The Cohomology Ring of \texorpdfstring{$X_{A_n}$}{XAn}}
In the case of type $A_n$, the torus-invariant prime divisors, which correspond to the rays of the fan $\Delta_{A_n}$, can be identified with the set of all non-empty proper subsets of~$[n+1]:=\{1,2,\dots,n+1\}$. Following~\cite{Abe2015}, we identify the set of rays of the fan $\Delta_{A_n}$ with the set of all non-empty proper subsets of $[n+1]$. This identification is given by the bijection:
\begin{align*}
    \Phi^* &\longrightarrow 2^{[n+1]} \setminus \{\emptyset, [n+1]\} \\
    u\w^{\vee}_i &\longmapsto \{u(1), \dots, u(i)\}
\end{align*}
    for $u \in \S_{n+1}$ and $1 \le i \le n$. We denote by $\tau_{S} \in H^2(X_{A_n}; \Z)$ the cohomology class of the torus-invariant prime divisor corresponding to a subset $S$.
The integer cohomology ring $H^*(X_{A_n}; \Z)$ is isomorphic to the quotient of the polynomial ring in the variables $\{\tau_{S}\}$ by the ideal generated by two types of relations:
\begin{enumerate}
  \item The Stanley-Reisner relations: 
  \begin{align}
  \label{eq:Stanley-Reisner-ideal}
   \tau_{S_1} \cdots \tau_{S_q} = 0 \qquad \text{if the subsets $S_1, \dots, S_q$ do not form a chain under inclusion.}
  \end{align}
    \item The linear relations: For any root $\alpha \in \Phi$, we have the relation $\sum_{x \in \Phi^*} \langle x, \alpha \rangle \tau_x = 0$. For the root system of type $A_n$, taking a root $\alpha = t_k - t_l$ for each $k,\ell\in[n+1], k \neq \ell$ translates this into a purely combinatorial relation among the $\tau_S$ classes~\cite[Eq.(6)]{Abe2015}:
    \[
   \sum_{\substack{\emptyset\subsetneq S\subsetneq[n+1] \\ k\in S, l\notin S}} \tau_S - \sum_{\substack{\emptyset\subsetneq S\subsetneq[n+1] \\ k\notin S, l\in S}} \tau_S = 0\qquad\text{for each }k,\ell\in[n+1].
    \]
\end{enumerate}

\subsection{Chern Classes of the Permutohedral Variety \texorpdfstring{$X_{A_n}$}{XAn}}
For an $n$-dimensional smooth toric variety $X$ with a fan~$\Delta$, the $k$-th Chern class of its tangent bundle $T_X$ is given by the sum of the cohomology classes of the orbit closures $X_{\sigma}$ corresponding to the $k$-dimensional cones $\sigma$~\cite{Fulton1993}:
\[
c_k(T_X) = \sum_{\sigma \in \Delta(k)} [X_\sigma],
\]
where $\Delta(k)$ denotes the set of $k$-dimensional cones of the fan $\Delta$.

For the permutohedral variety $X_{A_n}$, a $k$-dimensional cone $\sigma \in \Delta_{A_n}(k)$ corresponds to a chain of nested subsets $S_1 \subsetneq S_2 \subsetneq \dots \subsetneq S_k$ of $[n+1]$. The cohomology class of its orbit closure is given by the product of the corresponding divisor classes: 
\begin{align}
\label{eq:divisor-corresponding-k-cones}
[X_\sigma] = \tau_{S_1} \cdots \tau_{S_k}. 
\end{align}

The total Chern class of $T_{X_{A_n}}$ is given by a simple formula due to Klyachko~\cite{Klyachko1995}:
\[
c(T_{X_{A_n}}) = \prod_{i=1}^n (1+[\alpha_i]),
\]
where each class $[\alpha_i]$ is the sum of all torus-invariant divisor classes $\tau_S\in H^2(X_{A_n};\Z)$ for which the subset $S$ has cardinality~$i$:
\begin{align}
\label{eq:alpha-tau}
[\alpha_i] = \sum_{\substack{S \subsetneq [n+1] \\ |S|=i}} \tau_S.
\end{align}
Expanding the product formula for the total Chern class, the $k$-th Chern class $c_k=c_k(T_{X_{A_n}})$ is given by the $k$-th elementary symmetric function in the classes $[\alpha_i]$:
\[
c_k = e_k([\alpha_1], \dots, [\alpha_n]) = \sum_{1 \le i_1 < \dots < i_k \le n} [\alpha_{i_1}] \cdots [\alpha_{i_k}].
\]
When we substitute the expression for each $[\alpha_{i_j}]$ from \eqref{eq:alpha-tau}, a typical term in the sum becomes:
\[
\left( \sum_{|S_1|=i_1} \tau_{S_1} \right) \cdots \left( \sum_{|S_k|=i_k} \tau_{S_k} \right).
\]
The Stanley-Reisner relations~\eqref{eq:Stanley-Reisner-ideal} state that a product $\tau_{S_1} \cdots \tau_{S_k}$ vanishes unless the subsets form a chain under inclusion. Since we have $i_1 < i_2 < \dots < i_k$, the only non-vanishing terms in the expansion are those for which $S_1 \subsetneq S_2 \subsetneq \dots \subsetneq S_k$. A key combinatorial property of the permutohedral fan is that each such product, which corresponds to the class of a $k$-dimensional cone, appears in the expansion of $c_k$ exactly once. Thus, Klyachko's formula correctly recovers the general one:
\[
c_k = e_k([\alpha_1],\dots,[\alpha_n]) = \sum_{S_1 \subsetneq \dots \subsetneq S_k} \tau_{S_1} \cdots \tau_{S_k} = \sum_{\sigma \in \Delta_{A_n}(k)} [X_\sigma].
\]
In particular, the top Chern class is given by the product
\[
c_n = [\alpha_1] \cdots [\alpha_n].
\]
The pairing of the top Chern class $c_n$ with the fundamental homology class $[X_{A_n}]$ yields the Euler characteristic of $X_{A_n}$, which is the order of the Weyl group $\S_{n+1}$, $(n+1)!$:
\begin{equation}
\label{eq:euler-char-pairing}
\langle c_n, [X_{A_n}] \rangle = (n+1)!.
\end{equation}

\begin{example}
\label{example:totalChernclass}
Let us compute the total Chern class $c(T_{X_{A_2}})$. The non-empty proper subsets of $\{1,2,3\}$ are
\[
\{1\},\{2\},\{3\},\{1,2\},\{1,3\},\{2,3\}.
\]
The subsets of cardinality one, $\{1\},\{2\},\{3\}$, form a Weyl group orbit, as do the subsets of cardinality two, $\{1,2\},\{1,3\},\{2,3\}$. The classes $\alpha_1$ and $\alpha_2$ are therefore given by the sums over these orbits:
\[
[\a_1]=\tau_{\{1\}}+\tau_{\{2\}}+\tau_{\{3\}}, \qquad [\a_2]=\tau_{\{1,2\}}+\tau_{\{1,3\}}+\tau_{\{2,3\}}.
\]
The total Chern class is the product $(1+[\a_1])(1+[\a_2])$. The product $[\a_1][\a_2]$ simplifies due to the Stanley-Reisner relations. A term $\tau_{S_1}\tau_{S_2}$ with $|S_1|=1$ and $|S_2|=2$ vanishes unless the corresponding rays span a 2-dimensional cone in the fan, which occurs if and only if $S_1 \subsetneq S_2$. For instance, the term $\tau_{\{1\}}\tau_{\{2,3\}}$ vanishes because $\{1\} \not\subsetneq \{2,3\}$. The non-vanishing terms are precisely those of the form $\tau_{\{i\}}\tau_{\{i,j\}}$. Thus, the full expansion of the total Chern class is
\begin{align*}
c(T_{X_{A_2}}) = 1 &+ (\tau_{\{1\}}+\tau_{\{2\}}+\tau_{\{3\}}+\tau_{\{1,2\}}+\tau_{\{1,3\}}+\tau_{\{2,3\}}) \\
&+ (\tau_{\{1\}}\tau_{\{1,2\}}+\tau_{\{1\}}\tau_{\{1,3\}} +\tau_{\{2\}}\tau_{\{1,2\}}+\tau_{\{2\}}\tau_{\{2,3\}} +\tau_{\{3\}}\tau_{\{1,3\}}+\tau_{\{3\}}\tau_{\{2,3\}}).
\end{align*}
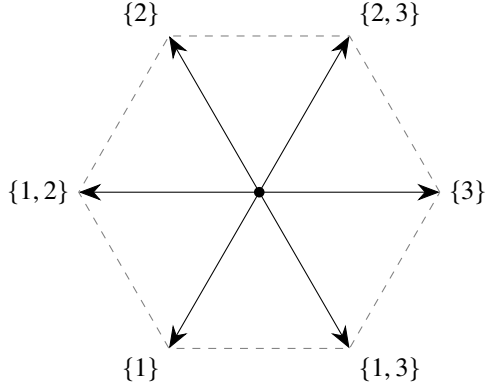
\begin{figure}[H]
\centering
\begin{tikzpicture}[scale=2]
\coordinate (O) at (0,0);
\coordinate (r1) at (240:1.2); 
\coordinate (r2) at (120:1.2); 
\coordinate (r3) at (0:1.2);
\coordinate (r23) at (60:1.2); 
\coordinate (r13) at (300:1.2); 
\coordinate (r12) at (180:1.2); 
\draw[-{Stealth[length=3mm]}] (O) -- (r1) node[anchor=north east] {$\{1\}$};
\draw[-{Stealth[length=3mm]}] (O) -- (r2) node[anchor=south east] {$\{2\}$};
\draw[-{Stealth[length=3mm]}] (O) -- (r3) node[anchor=west] {$\{3\}$};
\draw[-{Stealth[length=3mm]}] (O) -- (r23) node[anchor=south west] {$\{2,3\}$};
\draw[-{Stealth[length=3mm]}] (O) -- (r13) node[anchor=north west] {$\{1,3\}$};
\draw[-{Stealth[length=3mm]}] (O) -- (r12) node[anchor=east] {$\{1,2\}$};
\draw[dashed, thin, gray] (r3) -- (r23) -- (r2) -- (r12) -- (r1) -- (r13) -- cycle;
\fill (O) circle (1pt);
\end{tikzpicture}
\caption{The fan of the permutohedral variety $X_{A_2}$. Each ray is labeled by the corresponding non-empty proper subset of $\{1,2,3\}$.}
\label{fig:fan_A2}
\end{figure}
\end{example}

\subsection{Geometric Origin of the Cohomology Classes}
The cohomology classes $[\a_i]$ central to this paper have a geometric origin arising from the theory of flag varieties. The permutohedral variety $X_{A_n}$ can be realized as the closure of a generic orbit of the maximal torus $T_n$ of diagonal matrices in $\SL{n+1}$ acting on the full flag variety $\Fl{n+1} = \SL{n+1}/B$~\cite{Klyachko1985, Klyachko1995}, where $B$ is the Borel subgroup of upper triangular matrices in $\SL{n+1}$. 

Let $\lambda:T_n\to\C^*$ be a weight. This character extends to the Borel subgroup $B$. Following Brion \cite{Brion2005}, we define the line bundle $L_\lambda$ associated with the weight $\lambda$ as the quotient
\[
L_\lambda := (\SL{n+1} \times\C) / B,
\]
where the right action of $B$ on the product is given by 
$$
(g,v) \cdot b := (gb, \lambda(b)^{-1}v) \quad\text{ for } g \in \SL{n+1}, v \in \C, b \in B.
$$
The class $[\alpha_i]$ is the first Chern class of the line bundle $L_{\a_i}$ on $\Fl{n+1}$ associated with the simple root~$\alpha_i$:
\begin{equation}
\label{eq:comologyclass-root}
[\a_i] = c_1(L_{\a_i}) \in H^2(\Fl{n+1}; \Q).
\end{equation}
The class $[\a_i]$, which was defined combinatorially in the previous section via \eqref{eq:alpha-tau}, is precisely the restriction of the first Chern class of the line bundle $L_{\a_i}$ to the subvariety $X_{A_n} \subset \Fl{n+1}$. That is,
\[
[\a_i] = c_1(L_{\a_i}|_{X_{A_n}}) \in H^2(X_{A_n}; \Q).
\]
Similarly, we define the classes $[\w_i]$ corresponding to the fundamental weights $\w_i$,
where the fundamental weights are characterized by the relations
\[
\langle \w_i, \alpha_j^\vee \rangle = \delta_{ij}, \qquad 1 \le i,j \le n,
\]
with $\w_i \in E$ and the coroots $\alpha_j^\vee \in E^*$ (see subsection 2.1).
We then set
\begin{equation}
\label{eq:cohomologyclass-weight}
[\w_i] := c_1\!\left(L_{\w_i}\big|_{X_{A_n}}\right) \in H^2(X_{A_n}; \mathbb{Q}).
\end{equation}

The rational cohomology ring $H^*(X_{A_n};\mathbb{Q})$ admits a natural action of the Weyl group~$\mathfrak{S}_{n+1}$. The structure of the invariant subring $H^*(X_{A_n}; \Q)^{\S_{n+1}}$ is well studied~\cite{Klyachko1995}. This subring is generated by the classes $[\w_1], \dots, [\w_n]$, which satisfy the quadratic relations
\begin{align}\label{eq:quadratic-relations-prelim}
 [\alpha_i][\w_i] = 0 \qquad \text{for } 1 \le i \le n.
\end{align}

The relation $\alpha_i = -\w_{i-1} + 2\w_i - \w_{i+1}$ in the weight lattice corresponds to an isomorphism of line bundles:
\[
L_{\alpha_i} \cong L_{\w_{i-1}}^\vee \otimes L_{\w_i}^{\otimes 2} \otimes L_{\w_{i+1}}^\vee,
\]
where $L^{\vee}$ denotes the dual of the line bundle $L$.
For a nonsingular projective toric variety such as $X_{A_n}$, the first Chern class map $c_1$ induces a group isomorphism from the character lattice of the maximal torus $T_n$ to the second cohomology group $H^2(\Fl{n+1};\Z)$. This isomorphism translates the tensor product structure of line bundles into the additive structure of cohomology. Consequently, the isomorphism of line bundles above yields the following linear relations in $H^2(X_{A_n}; \mathbb{Q})^{\S_{n+1}}$:
\begin{align}\label{eq:Cartan-matrix-relation-prelim}
 [\alpha_i] = -[\w_{i-1}] + 2[\w_i] - [\w_{i+1}] \qquad\text{for }1\le i\le n,
\end{align}
where we adopt the convention $[\w_0]=[\w_{n+1}]=0$. For readers unfamiliar with the theory of flag varieties, the classes $[\w_i]$ can be defined in terms of the classes $[\a_j]$ as the unique solution to the system of linear equations given by \eqref{eq:Cartan-matrix-relation-prelim}.

\begin{remark}
As seen in Subsection 2.2, the cohomology classes $[\alpha_i]$ can also be expressed in terms of the classes $\tau_S \in H^*(X_{A_n}; \mathbb{Z})$. However, calculating Chern classes directly using the $\tau_S$ presentation is more complicated due to the large number of generators and relations involved (see~\autoref{example:totalChernclass}). On the other hand, the cohomology classes $\alpha_i$ and $\omega_i$ satisfy useful algebraic relations such as \eqref{eq:quadratic-relations-prelim} and \eqref{eq:Cartan-matrix-relation-prelim}. Therefore, in what follows, we will primarily work with the classes $[\alpha_i]$ and $[\omega_i]$ within the invariant subring~$H^*(X_{A_n}; \mathbb{Q})^{\S_n}$.    
\end{remark}

\section{Combinatorial Structure and Enumeration of Monomials}
\subsection{Fundamental Block Decomposition}

For the remainder of this section, we extend the definition of the simple roots $\alpha_i$ for all integers $i$. 
Let $T_n$ be the maximal torus of diagonal matrices in $\SL{n+1}$, and let 
$$
P(A_n) = \mathrm{Hom}(T_n, \mathbb{C}^*)
$$
be its weight lattice. 
The simple roots of type $A_n$ are elements of this lattice, given by the formula
\[
\alpha_i = -\w_{i-1} + 2\w_i - \w_{i+1} \qquad \text{for } 1 \le i \le n,
\]
with the convention that $\w_0=0$ and $\w_{n+1}=0$. We extend this definition for all integers $i$ as follows:
\begin{align}
\label{eq:extended-varphi-alpha}
\alpha_i &:=
\begin{cases}
-\w_{i-1}+2\w_i-\w_{i+1}, & 1\leq i\leq n, \\[4pt]
-\w_1, & i=0, \\[4pt]
-\w_n, & i=n+1, \\[4pt]
0, & i<0 \text{ or } i>n+1.
\end{cases}
\end{align}
This extension is naturally motivated by the standard group embedding 
$$
\iota: \SL{n+1} \hookrightarrow \SL{n+3}\qquad A \mapsto \mathrm{diag}(1, A, 1).
$$
This embedding restricts to an inclusion of the maximal torus, $\iota|_{T_n}: T_n \hookrightarrow T_{n+2}$, where $T_{n+2}$ is the maximal torus in $\SL{n+3}$. This inclusion, in turn, yields a surjective restriction map on the corresponding weight lattices:
\[
\mathrm{res}: P(A_{n+2}) \to P(A_n), \quad \text{defined by} \quad \mathrm{res}(\lambda) := \lambda \circ (\iota|_{T_n}).
\]
Let $\w'_0, \dots, \w'_{n+1}$ be the fundamental weights of type $A_{n+2}$. The fundamental weights of type $A_n$ are identified with the images of those of type $A_{n+2}$ under this map: $\w_i = \mathrm{res}(\w'_i)$ for $1 \leq i \leq n$. The kernel of this map is generated by $\w'_0$ and $\w'_{n+1}$.
Calculating the images of the simple roots $\a_i'~(0\le i\le n+1)$ of type $A_{n+2}$ under this restriction map gives:
\begin{itemize}
    \item $\mathrm{res}(\a'_i) = -\w_{i-1} + 2\w_i - \w_{i+1} = \alpha_i$ for $1 \leq i \leq n$.
    \item $\mathrm{res}(\a'_0) = -\w_1$ and $\mathrm{res}(\a'_{n+1}) = -\w_n$.
\end{itemize}
It is precisely this correspondence that makes sense of our extended definition of $\alpha_i$ in \eqref{eq:extended-varphi-alpha}. The same holds for the corresponding cohomology classes.

As established in Section 2, we identify elements of the weight lattice with their corresponding second cohomology classes. From now on, by an abuse of notation, we will denote the classes $[\alpha_i]$ and $[\varpi_i]$ simply by $\alpha_i$ and $\varpi_i$ respectively.

\begin{proposition}[Fundamental Block Decomposition]
\label{prop:fundamental-decomposition}
The exponent vector of any non-vanishing monomial in the expansion of $c_k c_{n-k}$ can be expressed as a sequence of the following fundamental blocks:
\begin{enumerate}[(i)]
    \item The block $(1)$.
    \item The block $(2,0)$.
    \item The block $(0,2)$.
\end{enumerate}
\end{proposition}
In what follows, we use the symbol $\star$ to denote any non-zero entry, i.e., $\star \in \{1,2\}$. To handle boundary conditions uniformly, we consider the augmented exponent vector $(v_0, v_1, \dots, v_n, v_{n+1})$, where the virtual entries $v_0$ and $v_{n+1}$ are defined to be of type $\star$. This convention corresponds to the extended definition of $\a_i$ in \eqref{eq:extended-varphi-alpha}.

For a sub-vector $S$ of an exponent vector, we let $N_j(S)$ denote the number of entries in $S$ equal to $j$. When $S$ is the entire exponent vector $v$, we simply write $N_j$ for $N_j(v)$.

\begin{lemma}[Vanishing Conditions for Exponent Vectors]
\label{lem:vanishing-conditions}
A monomial in the expansion of $c_k c_{n-k}$ vanishes if its augmented exponent vector contains a sub-vector of the form $(\star, S, \star)$, where $S$ is an alternating sub-vector of $0$s and $2$s satisfying $N_2(S) > N_0(S)$.
\end{lemma}

\begin{proof}[Proof of~\autoref{lem:vanishing-conditions}]
An alternating sub-vector $S$ of $0$s and $2$s satisfying $N_2(S) > N_0(S)$ must necessarily begin and end with a 2. Thus, it must be of the form $(2, 0, 2, \dots, 0, 2)$. The statement is therefore equivalent to proving the algebraic relation
\[
\a_{j-1} \left( \prod_{l=0}^{k} \a_{j+2l}^2 \right) \a_{j+2k+1} = 0
\]
for some integer $k \ge 0$. We prove this by induction on $k$.

For the base case $k=0$, which corresponds to the shortest such sub-vector $S=(2)$, the relation reduces to $\a_{j-1}\a_j^2\a_{j+1}=0$. This holds for all valid indices $j$. Specifically, for $1 < j < n$:
\begin{align*}
\alpha_{j-1} \alpha_j^2 \alpha_{j+1}
&= (\alpha_{j-1} \alpha_j \alpha_{j+1}) \cdot \alpha_j \\
&= (\alpha_{j-1} \alpha_j \alpha_{j+1}) (-\w_{j-1} + 2\w_j - \w_{j+1}) \\
&= -\alpha_j \alpha_{j+1}(\alpha_{j-1}\w_{j-1}) + 2\alpha_{j-1} \alpha_{j+1}(\alpha_j\w_j) - \alpha_{j-1}\alpha_j(\alpha_{j+1}\w_{j+1}) \\
&= 0.
\end{align*}
The final equality holds because each term contains a factor of the form $\alpha_i \w_i$, which vanishes by the relation~\eqref{eq:quadratic-relations-prelim}. The boundary cases $j=1$ and $j=n$ follow from the extended definition of $\a_i$, ensuring the relation holds universally.

Assume the statement holds for $k-1~(k \ge 1)$. Let $M_k$ be the expression on the left-hand side for the case~$k$:
\[
M_k = \a_{j-1} \a_j^2 \a_{j+2}^2 \cdots \a_{j+2k}^2 \a_{j+2k+1}.
\]
We apply the expansion $\a_j = -\w_{j-1} + 2\w_j - \w_{j+1}$ to one of the $\a_j$ factors in the initial part of the expression, $\a_{j-1} \a_j^2\a_{j+2}^2$. It follows that
\begin{align*}
\a_{j-1} \a_j^2\a_{j+2}^2
&=\a_{j-1}\a_j(-\w_{j-1} + 2\w_j - \w_{j+1})\a_{j+2}^2\\
&=-\a_{j-1}\a_j\w_{j+1}\a_{j+2}^2\\
&=-\frac{1}{2}\a_{j-1}\a_j(\a_{j+1}+\w_j+\w_{j+2})\a_{j+2}^2\\
&=-\frac{1}{2}\a_{j-1}\a_j\a_{j+1}\a_{j+2}^2.
\end{align*}
Substituting this result into the expression for $M_k$, we find that $M_k$ is proportional to
\[
(\a_{j-1}\a_j\a_{j+1}\a_{j+2}^2)(\a_{j+4}^2 \cdots \a_{j+2k}^2 \a_{j+2k+1}).
\]
The factor $\a_{j+1}\a_{j+2}^2 \cdots \a_{j+2k}^2 \a_{j+2k+1}$ corresponds to the case $k-1$ of our relation. By the inductive hypothesis, this factor vanishes. Consequently, $M_k$ vanishes. This completes the induction.
\end{proof}

\begin{proof}[Proof of~\autoref{prop:fundamental-decomposition}]
Let $v = (v_1, \dots, v_n)$ be the exponent vector of a non-vanishing monomial. First, we establish a necessary relation between the number of $0$s and $2$s. The total degree of the monomial is $n$, so $\sum v_i = N_1 + 2N_2 = n$. The total number of components is also $n$, so $N_0 + N_1 + N_2 = n$. Equating these two expressions for $n$, we have $N_1 + 2N_2 = N_0 + N_1 + N_2$, which simplifies to
\begin{equation}
\label{eq:two-equal-zero}
N_0 = N_2.
\end{equation}

To prove the proposition, we decompose the exponent vector $v$ into sub-vectors consisting only of $1$s and maximal alternating sub-vectors of $0$s and $2$s. We call these sub-vectors $1$-sequence and $02$-sequence respectively. We will show that for a non-vanishing monomial, any $02$-sequence $S$ in its exponent vector $v$ must satisfy the condition $N_2(S) = N_0(S)$.
Assume, for the sake of contradiction, that a non-vanishing monomial has an exponent vector~$v$ containing a $02$-sequence $S$ such that $N_2(S) > N_0(S)$. Since $S$ is a maximal alternating sub-vector, this inequality implies that $S$ must be of the form $S=(2,0,2,0,\dots,2)$. By the maximality of $S$, its adjacent entries in the augmented exponent vector must be of type~$\star$. This creates a sub-vector of the form $(\star, S, \star)$ in the augmented exponent vector. By Lemma~\ref{lem:vanishing-conditions}, this implies the monomial vanishes, which is a contradiction. Therefore, no $02$-sequence $S$ with $N_2(S) > N_0(S)$ can exist in the exponent vector of a non-vanishing monomial. Combined with the global condition $N_2 = N_0$, this forces every $02$-sequence $S$ to satisfy $N_2(S) = N_0(S)$. Such a $02$-sequence must be of the form $(0, 2, \dots, 0, 2)$ or $(2, 0, \dots, 2, 0)$. The former can be decomposed into blocks of $(0, 2)$, while the latter can be decomposed into blocks of $(2, 0)$. This completes the proof of the proposition.
\end{proof}

\subsection{Reduction Formulas}
Any monomial can be transformed into a rational multiple of the square-free monomial $c_n = \a_1\cdots\a_n$ by applying the reduction formulas provided in this subsection. 

\begin{lemma}[Reduction Formula 1]
\label{lem:reduction-rule1}
Replacing a local sub-vector of the form $(\star, 2, 0, \star)$ or $(\star, 0, 2, \star)$ in the augmented exponent vector with $(\star, 1, 1, \star)$ amounts to multiplying the monomial by a factor of $-1/2$.
\end{lemma}

\begin{proof}
It suffices to prove the corresponding algebraic relation, for instance,
\[
\a_{i-1} \a_i^2 \a_{i+2} = -\frac{1}{2} \a_{i-1} \a_i \a_{i+1} \a_{i+2}.
\]
The calculation relies on the relation $\a_i = -\w_{i-1} + 2\w_i - \w_{i+1}$ and the vanishing of products of the form $\a_j\w_j=0$, as established in \eqref{eq:quadratic-relations-prelim}. A direct calculation then shows that
\begin{align*}
\a_{i-1} \a_i^2 \a_{i+2}
&= \a_{i-1} \a_i (-\w_{i-1} + 2\w_i - \w_{i+1}) \a_{i+2} \\
&= -(\a_{i-1}\a_i\a_{i+2})\w_{i+1} \\
&= -\frac{1}{2}\a_{i-1}\a_i\a_{i+1}\a_{i+2}.
\end{align*}
The last equality follows from the relation $\w_{i+1}=\tfrac{1}{2}(\w_i+\a_{i+1}+\w_{i+2})$. The symmetric case is proved similarly.
\end{proof}

\begin{lemma}[Reduction Formula 2]
\label{lem:reduction-rule2}
Replacing a local sub-vector of the form $(\star, 2, 0, 0, 2, \star)$ in the augmented exponent vector with $(\star, 1, 1, 1, 1, \star)$ amounts to multiplying the monomial by a factor of $1/3$.
\end{lemma}
\begin{proof}
The proof relies on showing the algebraic relation
\[
\a_{i-1} \a_i^2 \a_{i+3}^2\a_{i+4}
= \frac{1}{3} \a_{i-1} \a_i \a_{i+1}\a_{i+2}\a_{i+3}\a_{i+4}.
\]
We first expand the squared factors using the relation $\a_j = -\w_{j-1} + 2\w_j - \w_{j+1}$ and the property $\a_j\w_j=0$. This yields
\begin{align*}
\a_{i-1} \a_i^2 \a_{i+3}^2\a_{i+4}
&= \a_{i-1}\a_i(-\w_{i-1}+2\w_i-\w_{i+1})\a_{i+3}(-\w_{i+2}+2\w_{i+3}-\w_{i+4})\a_{i+4} \\
&= \a_{i-1}\a_i(-\w_{i+1})\a_{i+3}(-\w_{i+2})\a_{i+4} \\
&= \a_{i-1}\a_i\a_{i+3}\a_{i+4} \cdot \w_{i+1}\w_{i+2}.
\end{align*}
Here, we substitute the expressions for $\w_{i+1}$ and $\w_{i+2}$ to obtain:
\begin{align*}
\a_{i-1} \a_i^2 \a_{i+3}^2\a_{i+4}
&= \a_{i-1}\a_i\a_{i+3}\a_{i+4} \cdot \frac{1}{4}(\w_i+\a_{i+1}+\w_{i+2})(\w_{i+1}+\a_{i+2}+\w_{i+3}) \\
&= \frac{1}{4} \a_{i-1}\a_i\a_{i+3}\a_{i+4} (\a_{i+1}+\w_{i+2})(\w_{i+1}+\a_{i+2}) \\
&= \frac{1}{4} \a_{i-1}\a_i\a_{i+3}\a_{i+4} (\a_{i+1}\a_{i+2} + \a_{i+1}\w_{i+1} + \w_{i+2}\a_{i+2} + \w_{i+1}\w_{i+2}) \\
&= \frac{1}{4} \a_{i-1}\a_i\a_{i+3}\a_{i+4} (\a_{i+1}\a_{i+2} + \w_{i+1}\w_{i+2}).
\end{align*}
In the second line, the terms $\w_i$ and $\w_{i+3}$ vanish due to multiplication by $\a_i$ and $\a_{i+3}$, respectively. In the final line, the cross terms vanish because they contain factors of the form $\a_j\w_j=0$ (specifically, $\a_{i+1}\w_{i+1}$ and $\a_{i+2}\w_{i+2}$).

Using the relation derived in the first step, this becomes
\[
\a_{i-1} \a_i^2 \a_{i+3}^2\a_{i+4} = \frac{1}{4}\a_{i-1}\a_i\a_{i+1}\a_{i+2}\a_{i+3}\a_{i+4} + \frac{1}{4}\a_{i-1} \a_i^2 \a_{i+3}^2\a_{i+4}.
\]
Rearranging this equation yields the desired result.
\end{proof}

\subsection{Criterion for Non-Vanishing Monomials}

In subsection~3.1, we established that the exponent vector of any non-vanishing monomial must be composed of fundamental blocks $(1)$, $(2,0)$, and $(0,2)$. Here, we establish the converse, thereby providing a complete criterion for non-vanishing monomials. The key idea is that any monomial whose exponent vector has such a structure can be reduced to a non-zero multiple of the top-degree monomial $c_n = \alpha_1 \cdots \alpha_n$ by \autoref{lem:reduction-rule1} and \autoref{lem:reduction-rule2}.

\begin{proposition}[Criterion for Non-Vanishing Monomials]
\label{prop:non-vanishing-criterion}
A monomial is non-vanishing if and only if its exponent vector can be decomposed into blocks of type $(1)$, $(2,0)$, and $(0,2)$.
\end{proposition}

\begin{proof}
$(\Longrightarrow)$ This direction is a direct consequence of Proposition~\ref{prop:fundamental-decomposition}.

$(\Longleftarrow)$ Conversely, assume the exponent vector $v$ of a monomial can be decomposed into blocks of type $(1)$, $(2,0)$, and $(0,2)$. We will show that this vector can be fully reduced to the form $(1,1,\dots,1)$ by applying the reduction formulas from \autoref{lem:reduction-rule1} and \autoref{lem:reduction-rule2}. This implies the monomial is a non-zero scalar multiple of the non-vanishing term $c_n = \alpha_1 \cdots \alpha_n$, and thus is non-vanishing itself.

First, we identify any adjacent pairs of blocks of the form $(2,0)$ followed immediately by $(0,2)$, and merge them into a single $(2,0,0,2)$ block. After this operation, any remaining $(2,0)$ block must have a $\star$ type neighbor to its right. We then reduce all these remaining $(2,0)$ blocks sequentially from left to right. The leftmost $(2,0)$ block is guaranteed to have a $\star$ type neighbor on its left, allowing for its reduction to $(1,1)$ by Lemma~\ref{lem:reduction-rule1}. This process is repeated until all $(2,0)$ blocks are eliminated. A symmetric argument applies to all remaining $(0,2)$ blocks, which are reduced sequentially from right to left.

At this stage, the exponent vector is composed solely of $(1)$ blocks and the merged $(2,0,0,2)$ blocks. The neighbors of any $(2,0,0,2)$ block must be $\star$ types, as all original $(2,0)$ and $(0,2)$ blocks have been converted to $(1,1)$. Therefore, we can apply Lemma~\ref{lem:reduction-rule2} to reduce each $(2,0,0,2)$ block to $(1,1,1,1)$.

Since this systematic procedure reduces the exponent vector to $(1,1,\dots,1)$, and each reduction step introduces a non-zero scalar coefficient, the original monomial must be a non-zero multiple of $c_n$. This completes the proof.
\end{proof}

\begin{example}
Let us consider two examples of monomials appearing in the expansion of~$c_4c_6$.
\begin{enumerate}[(i)]
    \item Consider the monomial $\a_1^2\a_3\a_4^2\a_7^2\a_9^2\a_{10}$, with the exponent vector 
    $$
    v=(2, 0, 1, 2, 0, 0, 2, 0, 2, 1).
    $$
    First, we use \autoref{prop:non-vanishing-criterion} to check if this monomial is non-vanishing. To be non-vanishing, its vector must be decomposable into blocks of type (1), (2,0), and (0,2). For the vector $v$, we can make the following decomposition:
    \[
    (2,0), (1), (2,0), (0,2), (0,2), (1)
    \]
    Since this is a valid decomposition, the monomial is non-vanishing.

    Next, we show how to reduce this vector to $(1,1,\dots,1)$. The key idea is to group adjacent $(2,0)$ and $(0,2)$ blocks together. For this vector, we group the central $(2,0)$ and the first $(0,2)$ to form a single $(2,0,0,2)$ block. So, for the reduction process, we view the arrangement like this:
    \[
    (2,0), (1), (2,0,0,2), (0,2), (1)
    \]
    From this arrangement, we first reduce the outer $(2,0)$ and $(0,2)$ blocks, followed by the inner $(2,0,0,2)$ block, according to the procedure described in the proof of \autoref{prop:non-vanishing-criterion}. Applying reduction $1$ and $2$ yields:
\[
\a_1^2\a_3\a_4^2\a_7^2\a_9^2\a_{10}=\bra{-\frac{1}{2}}^2\bra{\frac{1}{3}}\a_1\a_2\cdots\a_{10}.
\]

    \item Conversely, consider the monomial $\a_2 \a_3^2 \a_5^2 \a_7^2 \a_8 \a_{10}^2$, with the exponent vector $(0, 1, 2, 0, 2, 0, 2, 1, 0, 2)$.
    
    This monomial must vanish. Its exponent vector cannot be decomposed into the allowed blocks (1), (2,0), and (0,2), so it fails the condition in \autoref{prop:non-vanishing-criterion}.
    
    More directly, the reason it fails is that its augmented exponent vector contains the sub-vector $(1, 2, 0, 2, 0, 2, 1)$. This sub-vector has the form $(\star, S, \star)$ with $S=(2, 0, 2, 0, 2)$. As established in \autoref{lem:vanishing-conditions}, this structure forces the monomial to be zero.
\end{enumerate}
\end{example}

\subsection{Enumeration via Block Classification}

As established in \autoref{prop:non-vanishing-criterion}, a non-vanishing monomial is characterized by a unique partition of its exponent vector into blocks of type $(1)$, $(2,0)$, and $(0,2)$. To enumerate monomials that reduce to the same multiple of $c_n$, we now introduce the following five types of blocks for classification:

\begin{description}
\item[$\Lbox$] The block $(2,0)$. This reduces with a factor of $-\tfrac{1}{2}$.
\item[$\Rbox$] The block $(0,2)$. This also reduces with a factor of $-\tfrac{1}{2}$.
\item[$\Mbox$] To ensure a unique partition, any sequence consisting of an $\Lbox$ immediately followed by an $\Rbox$ must be treated as a single indecomposable block. We call this the $\Mbox$ block. Applying Reduction Formula~2 shows it reduces with a factor of $\tfrac{1}{3}$.
\item[$\kbox$] An entry $(1)$ originating from a factor in $c_k$.
\item[$\nkbox$] An entry $(1)$ originating from a factor in $c_{n-k}$.
\end{description}

An exponent vector of any non-vanishing monomial corresponds to a unique arrangement of these five blocks. The problem of counting non-vanishing monomials in the expansion of $c_kc_{n-k}$ is transformed into counting these arrangements. So, we address this by defining a Raw-count and a Net-count.

Let $i$ be the number of squared factors in a monomial appearing in the expansion of~ $c_kc_{n-k}$, where $0\leq i\leq k$. Let $m, \ell,$ and $r$ be the numbers of $\Mbox$, $\Lbox$, and $\Rbox$ blocks, respectively. These numbers are related by the equation $2m+\ell+r=i$. We begin with the Raw-count, which enumerates all arrangements without structural restrictions.
\begin{definition}
We define $\Raw{m}{\ell}{r}(k,i)$ to be the total number of block arrangements for given values of $k$ and $i$.
\end{definition}

\begin{lemma}\label{lem:Raw_binom}
Let $\Raw{m}{\ell}{r}(k,i)$ be the number defined above. Then
\begin{flalign*}
\Raw{m}{\ell}{r}(k,i)
&= \binom{n-3m-\ell-r}{m}
\binom{n-4m-\ell-r}{\ell}
\binom{n-4m-2\ell-r}{r}
\binom{n-2i}{k-i}\\
&= \binom{m+\ell+r}{m,\ell,r}\binom{n-3m-\ell-r}{m+\ell+r}\binom{n-2i}{k-i}.
\end{flalign*}
Here, $\binom{m+\ell+r}{m,\ell,r}$ is the multinomial coefficient. Furthermore,
\begin{align}
\Raw{m}{\ell}{r}(k,i)=\Raw{m}{\ell}{r}(i,i)\times\binom{n-2i}{k-i}.
\label{eq:Rawki-Rawii}
\end{align}
\end{lemma}
\begin{proof}
We choose the positions for the blocks in the order of $\Mbox$,~$\Lbox$,~$\Rbox$, and $\kbox$. The $m$ positions for the $\Mbox$ blocks can be chosen from $n-3m-\ell-r$ available positions. The $\ell$ positions for the $\Lbox$ blocks can be chosen from $n-4m-\ell-r$ positions. The $r$~positions for the $\Rbox$ blocks can be chosen from $n-4m-2\ell-r$ positions. Finally, the $k-i$ positions for the $\kbox$ blocks can be chosen from the remaining $n-4m-2\ell-2r$ positions. Recalling that $2m+\ell+r=i$, we have $n-4m-2\ell-2r = n-2(2m+\ell+r) = n-2i$. This provides the first equality. The second equality is clear from the definition of the multinomial coefficient.
\end{proof}

The Criterion for non-vanishing monomials \autoref{prop:non-vanishing-criterion} establishes that the decomposition of an exponent vector of non-vanishing monomials is unique. This uniqueness requires that any sequence that could be interpreted as an $\Lbox$ block immediately followed by an $\Rbox$ block must be identified as a single $\Mbox$ block. Consequently, a valid block decomposition corresponding to a non-vanishing monomial cannot contain an adjacent $\Lbox\Rbox$ pair. This distinction motivates defining a Net-count for arrangements that satisfy this rule, as opposed to a Raw-count of all possible arrangements.

\begin{definition}
Let $\Net{m}{\ell}{r}(k,i)$ denote the total number of block arrangements in which no $\Lbox$ block is immediately followed by an $\Rbox$ block.
\end{definition}

\begin{lemma}\label{lem:Net-Raw1}
For integers $i, k$ satisfying $0\leq i\leq k$ and non-negative integers $m, \ell, r$ satisfying $2m+\ell+r=i$, the following relation holds:
\begin{align}
\Raw{m}{\ell}{r}(k,i)=\sum_{p=0}^{\minval{\ell}{r}}\binom{m+p}{p}\Net{m+p}{\ell-p}{r-p}(k,i).
\end{align}
\end{lemma}
\begin{proof}
Let $p$ be the number of adjacent $\Lbox$ and $\Rbox$ pairs that form $\Mbox$ blocks. Then $0\leq p\leq \minval{\ell}{r}$. If $p$ such pairs form $\Mbox$ blocks, we have a total of $m+p$ $\Mbox$ blocks. The number of ways to choose which $p$ of these $(m+p)$ $\Mbox$ blocks originated from the $\Lbox$ $\Rbox$ pairs is $\binom{m+p}{p}$. Summing over all possible values of $p$, we obtain the formula in the lemma.
\end{proof}

\begin{proposition}\label{prop:Net-to-Raw2}
Let $\Raw{m}{\ell}{r}(k,i)$ and $\Net{m}{\ell}{r}(k,i)$ be the numbers defined above. Then the following relation holds:
\begin{align}
\Net{m}{\ell}{r}(k,i)=\sum_{p=0}^{\minval{\ell}{r}}(-1)^p\binom{m+p}{p}\Raw{m+p}{\ell-p}{r-p}(k,i).
\label{eq:Net-expansion-of-Raws}
\end{align}
Furthermore, combining this result with \eqref{eq:Rawki-Rawii}, we have
\begin{align}
\Net{m}{\ell}{r}(k,i)
= \binom{n-2i}{k-i}\times\Net{m}{\ell}{r}(i,i).
\label{eq:Netki-Netii}
\end{align}
\end{proposition}

\begin{proof}
By the symmetry between $\Lbox$ and $\Rbox$, it is sufficient to prove the case where $\ell \leq r$. We proceed by induction on $\ell = \minval{\ell}{r}$.

\paragraph{Case $\ell=0$:}
Since there are no $\Lbox$ blocks, no adjacent \Lbox\ and \Rbox\ pairs can form an \Mbox. Therefore, by definition,
\[
\Net{m}{0}{r}(k,i) = \Raw{m}{0}{r}(k,i),
\]
and the statement holds for $\ell=0$.

\paragraph{Case $\ell=1$:} From \autoref{lem:Net-Raw1}, the following recurrence relation holds:
\[
\Net{m}{1}{r}(k,i) = \Raw{m}{1}{r}(k,i) - \binom{m+1}{1}\Net{m+1}{0}{r-1}(k,i).
\]
Noting that $\Net{m+1}{0}{r-1}(k,i) = \Raw{m+1}{0}{r-1}(k,i)$, we obtain
\[
\Net{m}{1}{r}(k,i) = \Raw{m}{1}{r}(k,i) - \binom{m+1}{1}\Raw{m+1}{0}{r-1}(k,i).
\]
Thus, the statement holds for $\ell=1$.

\paragraph{Case $\ell$:}
Assume that the proposition holds for all non-negative integers $q < \ell$. From \autoref{lem:Net-Raw1},
\[
\Net{m}{\ell}{r}(k,i)=\Raw{m}{\ell}{r}(k,i)-\sum_{p=1}^{\ell}\binom{m+p}{p}\Net{m+p}{\ell-p}{r-p}(k,i).
\]
For $p \geq 1$, we have $\ell-p < \ell$. Thus, by the induction hypothesis, the term $\Net{m+p}{\ell-p}{r-p}(k,i)$ on the right-hand side can be written as:
\[
\Net{m+p}{\ell-p}{r-p}(k,i)=\sum_{q=0}^{\ell-p}(-1)^q\binom{m+p+q}{q}\Raw{m+p+q}{\ell-p-q}{r-p-q}(k,i).
\]
Therefore,
\[
\Net{m}{\ell}{r}(k,i)
=\Raw{m}{\ell}{r}(k,i)-\sum_{p=1}^{\ell}\binom{m+p}{p}\left(\sum_{q=0}^{\ell-p}(-1)^q\binom{m+p+q}{q}\Raw{m+p+q}{\ell-p-q}{r-p-q}(k,i)\right).
\]
We now find the coefficient of a specific term $\Raw{m+j}{\ell-j}{r-j}$ (for $1\leq j\leq \ell$) in this sum. This term arises from pairs $(p,q)$ such that $p+q=j$ with $p \ge 1, q \ge 0$. Its coefficient is the sum:
\begin{align}
-\sum_{p=1}^{j}\binom{m+p}{p} \cdot (-1)^{j-p}\binom{m+p+(j-p)}{j-p} = \sum_{p=1}^{j} (-1)^{j-p+1}\binom{m+p}{p}\binom{m+j}{j-p}.
\label{eq:Raw-mj-lj-rj}
\end{align}
Using the identity $\binom{m+p}{p}\binom{m+j}{j-p}=\binom{j}{p}\binom{m+j}{j}$ for an integer $p$ with $1\leq p\leq j$, the coefficient becomes
\begin{align*}
\sum_{p=1}^{j} (-1)^{j-p+1}\binom{m+p}{p}\binom{m+j}{j-p}
&= \sum_{p=1}^{j}(-1)^{j-p+1}\binom{j}{p}\binom{m+j}{j}\\
&= \binom{m+j}{j} \sum_{p=1}^{j}(-1)^{j-p+1}\binom{j}{p} \\
&= \binom{m+j}{j} (-1)^{j+1} \sum_{p=1}^{j}(-1)^{-p}\binom{j}{p}.
\end{align*}
For $j \ge 1$, the binomial theorem implies $\sum_{p=0}^{j}(-1)^p\binom{j}{p} = (1-1)^j = 0$. From this, we have
\[
\sum_{p=1}^{j}(-1)^{p}\binom{j}{p} = \left(\sum_{p=0}^{j}(-1)^p\binom{j}{p}\right) - \binom{j}{0} = 0 - 1 = -1.
\]
Therefore, the coefficient of $\Raw{m+j}{\ell-j}{r-j}(k,i)$ is
\[
\binom{m+j}{j} (-1)^{j+1}(-1) = (-1)^{j+2}\binom{m+j}{j} = (-1)^j\binom{m+j}{j}.
\]
Since this holds for each $j=1,\dots,\ell$, we have
\begin{align*}
\Net{m}{\ell}{r}(k,i) &= \Raw{m}{\ell}{r}(k,i) + \sum_{j=1}^{\ell}(-1)^j\binom{m+j}{j}\Raw{m+j}{\ell-j}{r-j}(k,i) \\
&= \sum_{j=0}^{\ell}(-1)^j\binom{m+j}{j}\Raw{m+j}{\ell-j}{r-j}(k,i),
\end{align*}
which shows that the proposition holds for $\ell$.
\end{proof}

\section{Coefficient Formula}
\label{sec:coefficient_formula}
In Section 3, we introduced a method for enumerating non-zero monomials, grouping those that reduce to the same multiple of the square-free monomial $c_n$. Specifically, this enumeration used a correspondence between the exponent vector of a monomial with $i$ squared factors and an arrangement of five types of blocks. The coefficient resulting from the reduction of a monomial is determined by the types of blocks in its arrangement, specifically:
\begin{itemize}
    \item Each $\Lbox$ or $\Rbox$ block contributes a factor of $-1/2$ (\autoref{lem:reduction-rule1}).
    \item Each $\Mbox$ block contributes a factor of $1/3$ (\autoref{lem:reduction-rule2}).
\end{itemize}
In this section, we combine these combinatorial results and reduction factors to determine the contribution of monomials with $i$ squared terms to the total coefficient $\mu_k(n)$ in the main theorem.

\begin{definition}
We define $\Xval{k}{i}$ as the total coefficient from monomials with $i$ squared factors in the expansion of $c_kc_{n-k}$, after reduction to $c_n$.
\end{definition}

\begin{lemma}
    
\label{lem:Xki_formula_sec4}
Let $0\leq i\leq k\leq n$. The coefficient $\Xval{k}{i}$ is given by the following formula:
\begin{align}
\Xval{k}{i}=\sum_{2m+\ell+r=i}\bra{-\frac{1}{2}}^{\ell+r}\bra{\frac{1}{3}}^m\Net{m}{\ell}{r}(k,i),
\end{align}
where $m, \ell,$ and $r$ are the numbers of $\Mbox, \Lbox,$ and $\Rbox$ blocks, respectively, in the block arrangement corresponding to a monomial's exponent vector $v$.
\end{lemma}

\begin{proof}
If an exponent vector $v$ decomposes into a block decomposition consisting of $m$ $\Mbox$ blocks, $\ell$ $\Lbox$ blocks, and $r$ $\Rbox$ blocks, its reduction yields a coefficient of $\bra{-1/2}^{\ell+r}\bra{1/3}^m$. The term $\Net{m}{\ell}{r}(k,i)$ counts the number of such valid block arrangements, each corresponding to a unique non-vanishing monomial. Therefore, the total coefficient $\Xval{k}{i}$ is obtained by summing the products of these coefficients and their corresponding counts over all tuples $\{m,\ell,r\}$ satisfying $2m+\ell+r=i$.
\end{proof}

\begin{corollary}
\label{cor:Xki-simplified-sec4}
The coefficient $\Xval{k}{i}$ can be expressed in terms of $\Xval{i}{i}$ as follows:
\begin{align}
\Xval{k}{i} = \binom{n-2i}{k-i} \Xval{i}{i}.
\label{eq:Xki-Xii-sec4}
\end{align}
\end{corollary}

\begin{proof}
Since \eqref{eq:Netki-Netii} gives 
\[
\Net{m}{\ell}{r}(k,i)=\binom{n-2i}{k-i}\Net{m}{\ell}{r}(i,i),
\]
we substitute this into the formula from \autoref{lem:Xki_formula_sec4}:
\begin{align*}
\Xval{k}{i}
&= \sum_{2m+\ell+r=i}\bra{-\frac{1}{2}}^{\ell+r}\bra{\frac{1}{3}}^m\cbra{\binom{n-2i}{k-i}\Net{m}{\ell}{r}(i,i)} \\
&= \binom{n-2i}{k-i} \cbra{\sum_{2m+\ell+r=i}\bra{-\frac{1}{2}}^{\ell+r}\bra{\frac{1}{3}}^m\Net{m}{\ell}{r}(i,i)} \\
&= \binom{n-2i}{k-i} \Xval{i}{i}.
\end{align*}
This proves the corollary.
\end{proof}

This corollary shows that the problem of computing $\Xval{k}{i}$ reduces to finding an explicit formula for $\Xval{i}{i}$. It turns out that $\Xval{i}{i}$ has the following representation.
\begin{proposition}\label{prop:formula-Xii}
For an integer $i$ satisfying $0\leq i\leq k$, the following holds:
\begin{align}
\Xval{i}{i}=(-1)^i\sum_{j=0}^{\lfloor i/2 \rfloor}\bra{\frac{1}{12}}^j\binom{i-j}{j}\binom{n-i-j}{i-j}.
\end{align}
\end{proposition}

First, we prepare a lemma that will be used in the proof.
\begin{lemma}\label{lem:triCoef}
For non-negative integers $L, M, N$ satisfying $L+M=N$, the following identity holds:
\begin{flalign*}
\sum_{r=0}^L\binom{N}{M,L-r,r}=2^L\binom{N}{M}.
\end{flalign*}
\end{lemma}

\begin{proof}[Proof of~\autoref{lem:triCoef}]
We consider the number of ways to color $N$ distinct balls, where $M$ are colored red, and the remaining $L$ are colored either white or blue. This can be counted in two ways.
\paragraph{Method 1:} First, we choose $M$ balls to be colored red, which can be done in $\binom{N}{M}$ ways. For each of the remaining $L$ balls, there are two color choices (white or blue). Thus, the total number of ways is $2^L\binom{N}{M}$.
\paragraph{Method 2:} We can sum over the number of blue balls, say $r$, where $0 \le r \le L$. For a fixed $r$, the number of ways to choose $M$ red balls, $L-r$ white balls, and $r$ blue balls is given by the multinomial coefficient $\binom{N}{M, L-r, r}$. Summing over all possible values of $r$ gives the total number of ways.
Since both methods count the same quantity, the equality holds.
\end{proof}

\begin{lemma}\label{lem:sum_of_Raw}
For non-negative integers $L, M$ and an index $i$ satisfying $2M+L=i$, the following identity holds:
\begin{align}
    \sum_{\ell+r=L} \Raw{M}{\ell}{r}(i,i) = 2^L \binom{L+M}{M}\binom{n-3M-L}{L+M}.
\end{align}
\end{lemma}
\begin{proof}
Substituting $\ell=L-r$ (for $0\leq r \leq L$), the left-hand side becomes:
\begin{align*}
\sum_{\ell+r=L} \Raw{M}{\ell}{r}(i,i)
&= \sum_{r=0}^L \Raw{M}{L-r}{r}(i,i) \\
&= \sum_{r=0}^L \binom{L+M}{M,L-r,r}\binom{n-3M-L}{L+M} \quad (\text{by }\autoref{lem:Raw_binom}) \\
&= \binom{n-3M-L}{L+M} \sum_{r=0}^{L} \binom{L+M}{M,L-r,r} \\
&= 2^L\binom{L+M}{M} \binom{n-3M-L}{L+M}. \quad\text{(by \autoref{lem:triCoef} with $N=L+M$)}
\end{align*}
\end{proof}

\begin{proof}[Proof of \autoref{prop:formula-Xii}]
We prove the proposition by considering two cases based on the parity of $i$.
For simplicity, we write $\Net{m}{\ell}{r}$ for $\Net{m}{\ell}{r}(i,i)$ and $\Raw{m}{\ell}{r}$ for $\Raw{m}{\ell}{r}(i,i)$.

\paragraph{Case 1:}
Suppose $i = 2q$ is even. From the condition $2m+\ell+r=2q$, we have $0\leq m\leq q$. Thus, in the defining formula from \autoref{lem:Xki_formula_sec4}, we can rewrite the sum in terms of $j$ by setting $m=q-j$.
\[
\Xval{2q}{2q} = \sum_{j=0}^{q} \bra{\frac{1}{3}}^{q-j}\bra{-\frac{1}{2}}^{2j} \sum_{\ell+r=2j} \Net{q-j}{\ell}{r}
\]
Let us consider the inner sum $S_j := \sum_{\ell+r=2j} \Net{q-j}{\ell}{r}$. By \autoref{prop:Net-to-Raw2},
\begin{align*}
S_j &= \sum_{r=0}^{2j}\Net{q-j}{2j-r}{r}\\
&= \sum_{r=0}^{2j} \cbra{\sum_{p=0}^{\minval{2j-r}{r}}(-1)^p \binom{q-j+p}{p} \Raw{q-j+p}{2j-r-p}{r-p}}.
\end{align*}
We now interchange the order of summation. We have two conditions, $0\leq r\leq 2j$ and $0\leq p\leq\minval{2j-r}{r}$, which imply that $0\leq p\leq j$. For a fixed $p$, the range of $r$ becomes $p \le r \le 2j-p$. A change of variables in the inner sum then gives
\begin{align*}
S_j
=\sum_{p=0}^j(-1)^p\binom{q-j+p}{p}\sum_{r=0}^{2(j-p)}\Raw{q-j+p}{2(j-p)-r}{r}.
\end{align*}
Applying \autoref{lem:sum_of_Raw} with $M=q-(j-p)$ and $L=2(j-p)$, we get
\[
\sum_{r=0}^{2(j-p)}\Raw{q-j+p}{2(j-p)-r}{r}=2^{2(j-p)}\binom{q+(j-p)}{q-(j-p)}\binom{n-3q+(j-p)}{q+(j-p)}.
\]
Thus,
\begin{align*}
S_j
&=\sum_{p=0}^j(-1)^p\binom{q-(j-p)}{p}\cbra{2^{2(j-p)}\binom{q+(j-p)}{q-(j-p)}\binom{n-3q+(j-p)}{q+(j-p)}}\\
&=\sum_{p'=0}^{j}(-1)^{j-p'}\binom{q-p'}{j-p'}\cbra{2^{2p'}\binom{q+p'}{q-p'}\binom{n-3q+p'}{q+p'}}, \qquad (p'=j-p).
\end{align*}
From the above, we have
\begin{flalign*}
\Xval{2q}{2q}
&=\sum_{j=0}^q\bra{\frac{1}{3}}^{q-j}\bra{-\frac{1}{2}}^{2j}S_j\\
&=\sum_{j=0}^q\bra{\frac{1}{3}}^{q-j}\bra{-\frac{1}{2}}^{2j}\sum_{p'=0}^{j}(-1)^{j-p'}\binom{q-p'}{j-p'}\cbra{2^{2p'}\binom{q+p'}{q-p'}\binom{n-3q+p'}{q+p'}}.
\end{flalign*}
Now, we interchange the order of summation to sum over $p'$ first. The summation range $0 \le p' \le j \le q$ can be rewritten as $0 \le p' \le q$ and $p' \le j \le q$.
\begin{flalign*}
\Xval{2q}{2q}
&= \sum_{p'=0}^q \sum_{j=p'}^q \bra{\frac{1}{3}}^{q-j}\bra{\frac{1}{4}}^{j}(-1)^{j-p'}\binom{q-p'}{j-p'} 2^{2p'}\binom{q+p'}{q-p'}\binom{n-3q+p'}{q+p'}.
\end{flalign*}
Factoring out the terms that depend on $p'$ from the inner sum, we compute the sum over $j$.
\begin{flalign*}
\Xval{2q}{2q}
&= \sum_{p'=0}^q 2^{2p'}\binom{q+p'}{q-p'}\binom{n-3q+p'}{q+p'} \cbra{\sum_{j=p'}^q \bra{\frac{1}{3}}^{q-j}\bra{\frac{1}{4}}^{j}(-1)^{j-p'}\binom{q-p'}{j-p'}}.
\end{flalign*}
For the inner sum, let $s=j-p'$, so that $j=s+p'$ and $s$ ranges from $0$ to $q-p'$.
\begin{align*}
\sum_{s=0}^{q-p'} \bra{\frac{1}{3}}^{q-(s+p')}\bra{\frac{1}{4}}^{s+p'}(-1)^{s}\binom{q-p'}{s}
&= \bra{\frac{1}{3}}^{q-p'}\bra{\frac{1}{4}}^{p'}\sum_{s=0}^{q-p'}\binom{q-p'}{s}\bra{\frac{1}{3}}^{-s}\bra{\frac{1}{4}}^{s}(-1)^s \\
&= \bra{\frac{1}{3}}^{q-p'}\bra{\frac{1}{4}}^{p'}\sum_{s=0}^{q-p'}\binom{q-p'}{s}\bra{-\frac{3}{4}}^s \\
&= \bra{\frac{1}{3}}^{q-p'}\bra{\frac{1}{4}}^{p'}\bra{1-\frac{3}{4}}^{q-p'} \qquad\text{(by the binomial theorem)}\\
&= \bra{\frac{1}{3}}^{q-p'}\bra{\frac{1}{4}}^{p'}\bra{\frac{1}{4}}^{q-p'} = \bra{\frac{1}{12}}^{q-p'}\bra{\frac{1}{4}}^{p'}.
\end{align*}
Substituting this result back into the main equation,
\begin{align*}
\Xval{2q}{2q}
&= \sum_{p'=0}^q 2^{2p'}\binom{q+p'}{q-p'}\binom{n-3q+p'}{q+p'} \times \bra{\frac{1}{12}}^{q-p'}\bra{\frac{1}{4}}^{p'} \\
&= \sum_{p'=0}^q \bra{4^{p'}\cdot\bra{\frac{1}{4}}^{p'}} \bra{\frac{1}{12}}^{q-p'}\binom{q+p'}{q-p'}\binom{n-3q+p'}{q+p'} \\
&= \sum_{p'=0}^q\bra{\frac{1}{12}}^{q-p'}\binom{q+p'}{q-p'}\binom{n-3q+p'}{q+p'}.
\end{align*}
By letting $j'=q-p'$, we have $p'=q-j'$, and the terms become
\begin{align*}
\binom{q+p'}{q-p'} &= \binom{q+(q-j')}{q-(q-j')} = \binom{2q-j'}{j'} = \binom{i-j'}{j'}, \\
\binom{n-3q+p'}{q+p'} &= \binom{n-3q+(q-j')}{q+(q-j')} = \binom{n-2q-j'}{2q-j'} = \binom{n-i-j'}{i-j'}.
\end{align*}
Since $i=2q$, we have $\fl{i/2}=q$. Thus, we obtain the desired formula for the even case.
\begin{flalign*}
\Xval{2q}{2q}=\Xval{i}{i}
=\sum_{j'=0}^{\fl{i/2}}\bra{
\frac{1}{12}
}^{j'}\binom{i-j'}{j'}\binom{n-i-j'}{i-j'}.
\end{flalign*}

The odd case follows a similar argument, though we provide the full details to emphasize the necessary modifications.
\paragraph{Case 2:}
Suppose $i = 2q+1$ is odd.From the condition $2m+\ell+r=2q+1$, we have $0\leq m\leq q$. We set $m=q-j$ in the defining formula from \autoref{lem:Xki_formula_sec4} to rewrite the sum in terms of $j$.
\[
\Xval{2q+1}{2q+1} = \sum_{j=0}^{q} \bra{\frac{1}{3}}^{q-j}\bra{-\frac{1}{2}}^{2j+1} \sum_{\ell+r=2j+1} \Net{q-j}{\ell}{r}.
\]
Let $S'_j := \sum_{\ell+r=2j+1} \Net{q-j}{\ell}{r}$ be the inner sum. By \autoref{prop:Net-to-Raw2},
\begin{align*}
S'_j &= \sum_{r=0}^{2j+1}\Net{q-j}{2j+1-r}{r} \\
&= \sum_{r=0}^{2j+1} \cbra{\sum_{p=0}^{\minval{2j+1-r}{r}}(-1)^p \binom{q-j+p}{p} \Raw{q-j+p}{2j+1-r-p}{r-p}} \\
&= \sum_{p=0}^j(-1)^p\binom{q-j+p}{p}\sum_{r=p}^{2j+1-p}\Raw{q-j+p}{2j+1-r-p}{r-p}.
\end{align*}
For the inner sum, we apply the substitution $r'=r-p$. The range $p \le r \le 2j+1-p$ thus become $0 \le r' \le 2(j-p)+1$. This gives
\[
\sum_{r'=0}^{2(j-p)+1}\Raw{q-j+p}{2(j-p)+1-r'}{r'}.
\]
We apply \autoref{lem:sum_of_Raw} with $M=q-j+p$ and $L=2(j-p)+1$. Since $L+M=q+j-p+1$, this sum is equal to
\begin{align*}
&2^{2(j-p)+1}\binom{q+j-p+1}{q-j+p}\binom{n-3(q-j+p)-(2(j-p)+1)}{q+j-p+1}\\
=\ &2^{2(j-p)+1}\binom{q+(j-p)+1}{q-(j-p)}\binom{n-3q+(j-p)-1}{q+(j-p)+1}.
\end{align*}
Therefore, $S'_j$ is given by
\begin{align*}
S'_j
&= \sum_{p=0}^j(-1)^p\binom{q-(j-p)}{p} 2^{2(j-p)+1}\binom{q+(j-p)+1}{q-(j-p)}\binom{n-3q+(j-p)-1}{q+(j-p)+1}.
\end{align*}
Let $p'=j-p$, which means $p=j-p'$. Then,
\begin{align*}
S'_j
&= \sum_{p'=0}^{j}(-1)^{j-p'}\binom{q-p'}{j-p'} 2^{2p'+1}\binom{q+p'+1}{q-p'}\binom{n-3q+p'-1}{q+p'+1}.
\end{align*}
Substituting this into the expression for $\Xval{2q+1}{2q+1}$ and interchanging the order of summation, we get
\begin{align*}
\Xval{2q+1}{2q+1}
&= \sum_{j=0}^q \bra{\frac{1}{3}}^{q-j}\bra{-\frac{1}{2}}^{2j+1} \sum_{p'=0}^{j}(-1)^{j-p'}\binom{q-p'}{j-p'} 2^{2p'+1}\binom{q+p'+1}{q-p'}\binom{n-3q+p'-1}{q+p'+1} \\
&= \sum_{p'=0}^q (-1)^{-p'} 2^{2p'+1} \binom{q+p'+1}{q-p'}\binom{n-3q+p'-1}{q+p'+1} \\
& \qquad \qquad \times \bra{\sum_{j=p'}^q \bra{\frac{1}{3}}^{q-j}\bra{-\frac{1}{2}}^{2j+1} (-1)^j \binom{q-p'}{j-p'}}.
\end{align*}
Letting $s=j-p'$ in the inner sum over $j$,
\begin{align*}
&\sum_{s=0}^{q-p'} \bra{\frac{1}{3}}^{q-p'-s}\bra{-\frac{1}{2}}^{2(s+p')+1} (-1)^{s+p'} \binom{q-p'}{s} \\
=\ & \bra{\frac{1}{3}}^{q-p'} \bra{-\frac{1}{2}}^{2p'+1} (-1)^{p'} \sum_{s=0}^{q-p'} \binom{q-p'}{s} \bra{\frac{1}{3}}^{-s} \bra{-\frac{1}{2}}^{2s} (-1)^s \\
=\ & \bra{\frac{1}{3}}^{q-p'} \bra{-\frac{1}{2}}^{2p'+1} (-1)^{p'} \sum_{s=0}^{q-p'} \binom{q-p'}{s} \bra{- \frac{3}{4}}^s \\
=\ & \bra{\frac{1}{3}}^{q-p'} \bra{-\frac{1}{2}}^{2p'+1} (-1)^{p'} \bra{1-\frac{3}{4}}^{q-p'} = \bra{\frac{1}{12}}^{q-p'} \bra{-\frac{1}{2}}^{2p'+1} (-1)^{p'}.
\end{align*}
Substituting this back into the main expression,
\begin{align*}
\Xval{2q+1}{2q+1}
&= \sum_{p'=0}^q (-1)^{-p'} 2^{2p'+1} \binom{q+p'+1}{q-p'}\binom{n-3q+p'-1}{q+p'+1} \bra{\frac{1}{12}}^{q-p'} \bra{-\frac{1}{2}}^{2p'+1} (-1)^{p'} \\
&= \sum_{p'=0}^q \bra{2^{2p'+1} \cdot \bra{-\frac{1}{2}}^{2p'+1}} \bra{\frac{1}{12}}^{q-p'} \binom{q+p'+1}{q-p'}\binom{n-3q+p'-1}{q+p'+1} \\
&= -\sum_{p'=0}^q \bra{\frac{1}{12}}^{q-p'} \binom{q+p'+1}{q-p'}\binom{n-3q+p'-1}{q+p'+1}.
\end{align*}
Let $j'=q-p'$, so that $p'=q-j'$. Then
\begin{align*}
q+p'+1 &= q+(q-j')+1 = 2q+1-j' = i-j', \\
n-3q+p'-1 &= n-3q+(q-j')-1 = n-2q-j'-1 = n-(2q+1)-j' = n-i-j'.
\end{align*}
Since $i=2q+1$, we have $\fl{i/2}=q$. Thus, we obtain the desired formula for the odd case.
\begin{align*}
\Xval{2q+1}{2q+1}=\Xval{i}{i} &= -\sum_{j'=0}^{q} \bra{\frac{1}{12}}^{j'} \binom{i-j'}{j'} \binom{n-i-j'}{i-j'} \\
&= (-1)^i \sum_{j'=0}^{\fl{i/2}} \bra{\frac{1}{12}}^{j'} \binom{i-j'}{j'} \binom{n-i-j'}{i-j'}.
\end{align*}
\end{proof}

\section{Main Theorem}
In this final section, we determine the coefficient $\mu_k(n)$ in the product formula
\[ c_k c_{n-k} = \mu_k(n) c_n, \]
which is the main goal of this paper. Recall that we expanded the product $c_k c_{n-k}$ and classified the resulting monomials by the number of squared factors, $i$. In Section 4, we defined $\Xval{k}{i}$ as the sum of the coefficients contributed by all monomials with exactly $i$ squared factors. To obtain the total coefficient $\mu_k(n)$, we must therefore sum these contributions over all possible values of $i$, ranging from 0 to $k$. This leads to the relation:
\[
\mu_k(n)=\sum_{i=0}^k\Xval{k}{i}.
\]

\begin{theorem}
\label{thm:MainTheorem}
Let $k$ be an integer satisfying $0\leq k\leq n$. The coefficient $\muval{k}{n}$ in the relation $c_kc_{n-k}=~\muval{k}{n}c_n$ is given by
\begin{align}
\muval{k}{n}=\sum_{j=0}^{\floor{k/2}}\bra{\frac{1}{12}}^j\binom{k-j}{j}\binom{n-k-j}{j}.
\end{align}
\end{theorem}

\begin{example}
Before proceeding to the proof of~\autoref{thm:MainTheorem}, we show some examples for small values of $k$.
\begin{itemize}
    \item For $k=1$, the sum is over $j=0$ only.
    \begin{align*}
    \mu_1(n) &= \binom{1}{0}\binom{n-1}{0}
    \end{align*}

    \item For $k=2$, the sum is over $j=0,1$.
    \begin{align*}
    \mu_2(n) &= \binom{2}{0}\binom{n-2}{0} + \frac{1}{12}\binom{1}{1}\binom{n-3}{1}
    \end{align*}

    \item For $k=3$, the sum is over $j=0,1$.
    \begin{align*}
    \mu_3(n) &= \binom{3}{0}\binom{n-3}{0} + \frac{1}{12}\binom{2}{1}\binom{n-4}{1}
    \end{align*}

    \item For $k=4$, the sum is over $j=0,1,2$.
    \begin{align*}
    \mu_4(n) &= \binom{4}{0}\binom{n-4}{0} + \frac{1}{12}\binom{3}{1}\binom{n-5}{1} + \left(\frac{1}{12}\right)^2\binom{2}{2}\binom{n-6}{2}
    \end{align*}
\end{itemize}
\end{example}

\begin{proof}[Proof of~\autoref{thm:MainTheorem}]
We begin by recalling from \autoref{cor:Xki-simplified-sec4} the relation
\begin{flalign*}
\Xval{k}{i}=\Xval{i}{i}\times\binom{n-2i}{k-i}.
\end{flalign*}

By \autoref{prop:formula-Xii},
\begin{flalign*}
\Xval{i}{i}=(-1)^i\sum_{j=0}^{\floor{i/2}}\bra{\frac{1}{12}}^j\binom{i-j}{j}\binom{n-i-j}{i-j}.
\end{flalign*}
Therefore,
\begin{flalign}
\Xval{k}{i}
&= (-1)^i\sum_{j=0}^{\fl{i/2}}\bra{\frac{1}{12}}^j
\binom{i-j}{j}\binom{n-i-j}{i-j}\binom{n-2i}{k-i} \notag\\
&= (-1)^i\sum_{j=0}^{\fl{i/2}}\bra{\frac{1}{12}}^j
\binom{i-j}{j}\binom{k-j}{k-i}\binom{n-i-j}{k-j}.
\label{eq:formula-X_ki}
\end{flalign}
The second equality follows from the identity
\[
\binom{n-i-j}{i-j}\binom{n-2i}{k-i}=\binom{k-j}{k-i}\binom{n-i-j}{k-j}.
\]
Thus, $\mu_k(n)=\sum_{i=0}^k\Xval{k}{i}$ can be expressed as
\[
\mu_k(n)=\sum_{i=0}^k\cbra{(-1)^i\sum_{j=0}^{\fl{i/2}}\bra{\frac{1}{12}}^j
\binom{i-j}{j}\binom{k-j}{k-i}\binom{n-i-j}{k-j}}.
\]
We now interchange the order of summation. The summation ranges $0\le i\le k$ and $0 \le j \le \fl{i/2}$ are equivalent to $0 \le j \le \fl{k/2}$ and $2j\le i\le k$. Therefore,
\[
\mu_k(n)
= \sum_{j=0}^{\fl{k/2}} \bra{\frac{1}{12}}^j \left( \sum_{i=2j}^k (-1)^i \binom{i-j}{j}\binom{k-j}{k-i}\binom{n-i-j}{k-j} \right).
\]
To evaluate the inner sum, let $p=k-i$. As $i$ ranges from $2j$ to $k$, $p$ ranges from $k-2j$ to~$0$. Thus,
\begin{align}
\mu_k(n)
=\sum_{j=0}^{\fl{k/2}} \bra{\frac{1}{12}}^j \left(
\sum_{p=0}^{k-2j} (-1)^{k-p} \binom{k-p-j}{j}\binom{k-j}{p}\binom{n-k+p-j}{k-j}
\right).
\label{eq:Key-prop-in-MainTheorem}
\end{align}
The inner sum is evaluated by the following proposition.
\begin{proposition}\label{prop:Key-equality}
For integers $k,j,n$ satisfying $0\le j\le\lfloor k/2\rfloor$, the following identity holds:
\[
\sum_{p=0}^{k-2j} (-1)^{k-p} \binom{k-p-j}{j}\binom{k-j}{p}\binom{n-k+p-j}{k-j}
= \binom{k-j}{j}\binom{n-k-j}{j}.
\]
\end{proposition}

Before proving the proposition, we note the following identity.
\begin{lemma}
\label{lem:Chu-Vandermonde-Identity}
\[
\sum_{p=0}^u (-1)^p \binom{u}{p}\binom{s+p}{t}=(-1)^u\binom{s}{t-u}.
\]
\end{lemma}

\begin{proof}[Proof of~\autoref{lem:Chu-Vandermonde-Identity}]
Let $f(x)=\sum_{p\geq 0}a_px^p$ be a power series.
We denote the coefficient of $x^t$ in $f(x)$ by $[x^t]f(x)$.
Then $[x^t](1+x)^{s+p} = \binom{s+p}{t}$. It follows that
\begin{flalign*}
\sum_{p=0}^u (-1)^p \binom{u}{p}\binom{s+p}{t}
&= \sum_{p=0}^u (-1)^p \binom{u}{p}[x^t](1+x)^{s+p}\\
&= [x^t](1+x)^{s}\sum_{p=0}^u (-1)^p \binom{u}{p}(1+x)^{p}.
\end{flalign*}
The sum on the right is
\begin{align*}
\sum_{p=0}^u (-1)^p \binom{u}{p}(1+x)^{p}
= \cbra{1-(1+x)}^u=(-x)^u.
\end{align*}
Therefore,
\begin{flalign*}
\sum_{p=0}^u (-1)^p \binom{u}{p}\binom{s+p}{t}
&= [x^t](1+x)^{s}(-x)^u\\
&= (-1)^u[x^{t-u}](1+x)^s\\
&= (-1)^u\binom{s}{t-u}.
\end{flalign*}
\end{proof}

\begin{proof}[Proof of \autoref{prop:Key-equality}]
Let $S$ be the left-hand side of the identity in the proposition. To factor terms out of the summation, we use the identity $\binom{k-p-j}{j}\binom{k-j}{p} = \binom{k-j}{j}\binom{k-2j}{p}$. This holds because
\begin{align*}
\binom{k-j}{p}\binom{k-j-p}{j} = \frac{(k-j)!}{p!(k-j-p)!}\frac{(k-j-p)!}{j!(k-2j-p)!} = \frac{(k-j)!}{p!j!(k-2j-p)!},
\end{align*}
and
\begin{align*}
\binom{k-j}{j}\binom{k-2j}{p} = \frac{(k-j)!}{j!(k-2j)!}\frac{(k-2j)!}{p!(k-2j-p)!} = \frac{(k-j)!}{j!p!(k-2j-p)!}.
\end{align*}

Using this identity, we can rewrite $S$ as:
\begin{align*}
S &= (-1)^k \sum_{p=0}^{k-2j} (-1)^p \binom{k-j}{j}\binom{k-2j}{p}\binom{n-k+p-j}{k-j} \\
&= (-1)^k \binom{k-j}{j} \sum_{p=0}^{k-2j} (-1)^p \binom{k-2j}{p}\binom{n-k+p-j}{k-j}.
\end{align*}
Applying \autoref{lem:Chu-Vandermonde-Identity} with $s=n-k-j, ~t=k-j,$ and $u=k-2j$, the sum becomes
\[
\sum_{p=0}^{k-2j} (-1)^p \binom{k-2j}{p}\binom{n-k-j+p}{k-j} = (-1)^{k-2j}\binom{n-k-j}{k-j-(k-2j)} = (-1)^{k-2j}\binom{n-k-j}{j}.
\]
Therefore,
\[
S
= (-1)^k \binom{k-j}{j} \cdot (-1)^{k-2j}\binom{n-k-j}{j}
= (-1)^{2k-2j} \binom{k-j}{j}\binom{n-k-j}{j}
= \binom{k-j}{j}\binom{n-k-j}{j},
\]
which proves the proposition.
\end{proof}

Substituting the result of \autoref{prop:Key-equality} into the right-hand side of \eqref{eq:Key-prop-in-MainTheorem} gives
\[
\mu_k(n) = \sum_{j=0}^{\fl{k/2}} \bra{\frac{1}{12}}^j \binom{k-j}{j}\binom{n-k-j}{j}.
\]
This completes the proof of the main theorem.
\end{proof}

\begin{corollary}
\label{cor:Chern-number}
The Chern number $\langle c_k c_{n-k}, [X_{A_n}] \rangle$ of $X_{A_n}$ is given by
\[
\langle c_k c_{n-k}, [X_{A_n}] \rangle=(n+1)!\mu_k(n).
\]
\end{corollary}

\begin{proof}
From the relation $c_kc_{n-k}=\mu_k(n)c_n$, we take the pairing with the fundamental homology class $[X_{A_n}]$ (cf. \cite{MilnorStasheff1974}):
\begin{flalign*}
\langle c_kc_{n-k},[X_{A_n}] \rangle
&= \langle \mu_k(n)c_n,[X_{A_n}] \rangle\\
&= \mu_k(n) \langle c_n,[X_{A_n}] \rangle\\
&= \mu_k(n) \times (n+1)! \qquad \text{(by \eqref{eq:euler-char-pairing})}
\end{flalign*}
\end{proof}

\section*{Future Work}
A natural extension of this work is to investigate the general product
$$
c_\lambda := \prod_{i=1}^s c_{\lambda_i} = \mu_\lambda(n) c_n
\quad
\text{
for any partition $\lambda = (\lambda_1, \dots, \lambda_s) \vdash n$.
}
$$
We expect that \cite[Lemma 5.1]{AHKZ2024}, which concerns the $\w_i$ classes, will provide an effective tool to extend the reduction techniques used in this paper to the general case $c_\lambda$. This lemma played an important role in the development of the Peterson Schubert calculus in type $A$.

\section*{Acknowledgements}
The author would like to express his sincere gratitude to Hiraku Abe for suggesting the problem and for his constant encouragement and invaluable discussions. He is also deeply indebted to Professors Mikiya Masuda, Tatsuya Horiguchi, Takashi Sato, and Haozhi Zeng for their many insightful and valuable comments, which greatly improved the manuscript.


\begin{thebibliography}{9}

\bibitem{Abe2015}
H.~Abe, \emph{Young diagrams and intersection numbers for toric manifolds associated with {W}eyl chambers}, Electron. J. Combin. \textbf{22} (2015), no.~2, Paper 2.4, 24 pp.

\bibitem{AHKZ2024}
H.~Abe, T.~Horiguchi, H.~Kuwata, and H.~Zeng, 
\emph{Geometry of {Peterson} {Schubert} calculus in type {A} and left-right diagrams}, 
Algebraic Combin. \textbf{7} (2024), no.~2, 383--412. 

\bibitem{Brion2005}
M.~Brion, \emph{Lectures on the geometry of flag varieties}, in Topics in cohomological studies of algebraic varieties, Trends Math., Birkh\"{a}user, Basel, 2005, pp.~33--85.

\bibitem{DeMariProcesiShayman1992}
F.~De~Mari, C.~Procesi, and M.~A. Shayman, \emph{Hessenberg varieties}, Trans. Amer. Math. Soc. \textbf{332} (1992), no.~2, 529--534.

\bibitem{DolgachevLunts1994}
V.~Dolgachev and V.~Lunts, \emph{A character formula for the representation of a Weyl group in the cohomology of the associated toric variety}, J. Algebra \textbf{168} (1994), no.~3, 741--772.

\bibitem{Fulton1993}
W.~Fulton, \emph{Introduction to toric varieties}, Annals of Mathematics Studies, vol. 131, Princeton University Press, Princeton, NJ, 1993.

\bibitem{HuhKatz2012}
J.~Huh and E.~Katz, \emph{Log-concavity of characteristic polynomials and the {B}ergman fan of matroids}, Math. Ann. \textbf{354} (2012), no.~3, 1103--1116.

\bibitem{Klyachko1985}
A.~A. Klyachko, \emph{Orbits of a maximal torus on a flag space}, Funct. Anal. Appl. \textbf{19} (1985), no.~1, 65--66.

\bibitem{Klyachko1995}
A.~A. Klyachko, \emph{Toric varieties and flag varieties}, Proc. Steklov Inst. Math. \textbf{208} (1995), no.~1, 124--145.

\bibitem{MilnorStasheff1974}
J.~W. Milnor and J.~D. Stasheff, \emph{Characteristic classes}, Annals of Mathematics Studies, No. 76, Princeton University Press, Princeton, N. J.; University of Tokyo Press, Tokyo, 1974.

\bibitem{Procesi1990}
C.~Procesi, \emph{The toric variety associated to {W}eyl chambers}, in Mots, Lang. Raison. Calc., Herm\`{e}s, Paris, 1990, pp.~153--161.

\bibitem{Stembridge1994}
J.~Stembridge, \emph{Some permutation representations of {W}eyl groups associated with the cohomology of toric varieties}, Adv. Math. \textbf{106} (1994), no.~2, 244--301.

\end{thebibliography}
\end{document}